\documentclass{amsart}
\usepackage{amsmath}
\usepackage{amsthm} 
\usepackage{enumerate}
\usepackage{latexsym}
\usepackage{multicol}
\usepackage[T1]{fontenc}
\newcommand{\Lin}{\mathcal{L}}
\newcommand{\Seg}{\mathcal{S}}
\newcommand{\Pfin}{\mathcal{P}^{f}}
\newcommand{\Pinf}{\mathcal{P^{\infty}}}
\newcommand{\Fold}[1]{\mathcal{G}_{#1}}
\newcommand{\Rot}{\mathcal{R}}
\newcommand{\Ref}{\mathcal{V}}
\newcommand{\norm}[1]{\left\Vert#1\right\Vert}

\newcommand{\eps}{\varepsilon}

\newcommand{\E}{\mathcal{E}}
\newcommand{\F}{\mathcal{F}}
\newcommand{\Haus}{\mathcal{H}}
\newcommand{\Leb}{\mathcal{L}}
\newcommand{\Saut}{J}

\newcommand{\Hun}{\mathcal{H}^1}

\newcommand{\N}{\mathbb{N}}
\newcommand{\R}{\mathbb{R}}

\newcommand{\om}{\Omega}
\newcommand{\dom}{{\partial \Omega}}

\newcommand{\dys}{\displaystyle}

\newcommand{\interior}[1]{\stackrel{\circ}{#1}}

\newtheorem{thm}{Theorem}[section]

\newtheorem{theorem}[thm]{Theorem}

\newtheorem{lemma}[thm]{Lemma}

\newtheorem{proposition}[thm]{Proposition}
\newtheorem{definition}[thm]{Definition}
\theoremstyle{definition}

\theoremstyle{remark}

\newtheorem{remark}[thm]{Remark}

\newtheorem{example}[thm]{Example}
\numberwithin{equation}{section}

\title{A selection criterion of solutions to a system of eikonal equations}

\author{Gisella Croce and Giovanni Pisante}
\address{Laboratoire de Math\'ematiques Appliqu\'ees du Havre, Universit\'e du Havre, 25, rue Philippe Lebon, 76063 Le Havre (FRANCE)
\\ 
Dipartimento di Matematica, Seconda Universit\`a degli studi di Napoli, Via Vivaldi, 43, 81100 Caserta (ITALY)}
\email{gisella.croce@univ-lehavre.fr, pisante@unina.it}

\begin{document}
\maketitle

\section{Introduction}\label{sec-intro}
In this article we deal with the following system of eikonal equations:
\begin{equation}\label{SEE}
\left\{
\begin{array}{cl}
\displaystyle \left|\frac{\partial u}{\partial x_i}\right|=1\,,\,
&i=1,\dots, N,\,\, \mbox{a.e. in}\, \,\Omega
\vspace{0.3cm}\\
u=0\,, & \mbox{on} \,\,\, \partial \Omega\,,
\end{array}
\right. 
\end{equation}
where  $\Omega$ is an open bounded 
connected Lipschitz domain of  $\R^N$. 
System (\ref{SEE}) arises in several nonlinear models in mechanics
and material science and the problem of the existence of solutions has
been widely studied. With no attempt to be exhaustive here we recall
some results that have motivated our study.

Examining the {\it pyramidal construction}, introduced in
\cite{cellina} and \cite{friesecke}, one can easily see that there
exist infinitely many $W^{1,\infty}_0(\Omega)$ solutions of \eqref{SEE}.
Indeed, if $Q\subset \R^N$ is a hyperrectangle oriented in such a way
that the exterior
normal to its faces is contained in the
set $$E=\{x=(x_1,\dots,x_N)\in \R^N : |x_1|=\dots=|x_N|\},$$  then the
distance function in the $l^1-$norm from the boundary of $Q$,
$d_1(\cdot, \partial Q)$, solves problem (\ref{SEE}) in  $Q$. 
In a general domain $\Omega$, 
Vitali's theorem allows us to cover $\Omega$, up to a Lebesgue measure zero set,
by a countable union of domains $Q_i$ as before. Then the function defined by 
$d_1(\cdot, \partial Q_i)$ {in} $Q_i, i \in \N$, and 0 elsewhere,
is a $W^{1,\infty}(\Omega)$ solution to system (\ref{SEE}).
Since there are infinitely many Vitali's coverings of $\Omega$,
 problem (\ref{SEE}) admits  infinitely many $W^{1,\infty}(\Omega)$ solutions.
As a consequence, 
it is an interesting question to select and characterize 
a particular class of solutions or better one single solution.

Let $\mathfrak{S}(\Omega)$ denote the set of the
$W^{1,\infty}(\Omega)$ solutions to problem (\ref{SEE}).
One possible strategy to select a specific function in $\mathfrak{S}(\Omega)$
can be developed using the theory of viscosity solution, introduced
by Crandall and Lions (see \cite{Crandall-Lions,implicitdacomarc}).
In this approach one takes advantage of the various useful properties enjoyed by viscosity solutions, such as maximality, uniqueness and explicit formulas.
 To do that, note that system (\ref{SEE}) is  equivalent to
\begin{equation}
\left\{
\begin{array}{cl}
\displaystyle F(Du)=0\,,\,
&\mbox{a.e. in} \,\Omega
\vspace{0.1cm}\\
u=0\,, &\mbox{on}\,\, \partial \Omega\,,
\end{array}
\right.\label{viscointroduction}
\end{equation}
where $F:\R^N\to \R$ is any continuous function  which is zero if and only if  $|x_i|=1$ for every $i=1,\dots,N$.
One therefore investigates the existence of 
a viscosity solution to problem
(\ref{viscointroduction}) with any suitably chosen $F$. 
It is well known that for $N=1$, $u(x)=d_1(x,\partial \Omega)$ is the unique viscosity solution to problem (\ref{viscointroduction}).
In higher dimensions, some geometrical conditions on the domain come
into play. 
Indeed in \cite{C-D-G-G} and \cite{Pisante} it is proved that there exists a viscosity solution to (\ref{viscointroduction}) 
if and only if $\Omega$ is a 
hyperrectangle such that the normals to each
face lie in $E$.   
In this case, $d_{1}(\cdot,\partial \Omega)$ is a viscosity solution. 
Although the previous result is in some sense negative,
it was nonetheless put to use in \cite{dacorogna-marcellini}  to
select
a special solution to problem (\ref{SEE}) (and to more general problems). 
There the authors construct a Vitali covering of $\Omega$ made up of domains $\Omega_i$ 
admitting a viscosity solution and define the relative viscosity solution over each of these sets. The  covering is built in a recursive way, with the idea of taking, at every step, the largest possible hyperrectangle.

Another  strategy to characterize a class of functions in $\mathfrak{S}(\Omega)$ 
is to use a variational method, that is, to choose a meaningful
functional over $\mathfrak{S}(\Omega)$ and to optimize it. 
The minimizers or maximizers, supposing they exist, 
would be selected solutions to problem (\ref{SEE}).
There is an evident difficulty to apply this method, since the set $\mathfrak{S}(\Omega)$ is not convex.
For example the natural functionals 
$$
v \to \int_{\Omega}|v|^p\,,\quad p\geq 1, 
$$
have in general neither  a minimizer nor a maximizer 
over $\mathfrak{S}(\Omega)$. 
Indeed any minimizing sequence converges to 0, which does not belong to $\mathfrak{S}(\Omega),$ 
and the limit function of maximizing sequences is 
plus or minus the distance from $\partial \Omega$ in the $l^1$-norm, which usually do 
not verify (\ref{SEE}).

In \cite{dacorogna-glowinski-pan} we find the first attempt of selection through a variational method. The authors study 
numerically a variational problem over the set of non-negative 
solutions to problem (\ref{SEE}): 
they obtain a maximizing sequence for the problem
\begin{equation}\label{glow}
\sup\left\{\int_\Omega u, \quad u\geq 0,\,\, u \in \mathfrak{S}(\Omega) 
\right\}
\end{equation}
through the numerical minimization of 
$$
\mathcal{J}(u):= -\int_\Omega u + \frac{1}{2}\int_\Omega | \nabla u |^2 
+\frac{\eps}{2}\int_\Omega |\Delta u|^2 + \frac{1}{2\eps}
\sum_{i=1}^{N} \int_\Omega\left(\left|\frac{\partial u}{\partial x_i}\right| -1 \right)^2\,.
$$
Unfortunately, as we said above, there is in general no optimal solution 
to the variational problem (\ref{glow}). Nevertheless, according to the seminal idea presented in
\cite{dacorogna-glowinski-pan} of selecting "regular solutions", it seems to be
natural to minimize in some way the discontinuity set of the gradient
of the solutions. With this aim in mind let us consider 
the functions $v\in \mathfrak{S}(\Omega)$ such that  the distributional gradient of $\frac{\partial v}{\partial x_i}, i=1,\dots,N,$ has no Cantor part locally, that is, $\frac{\partial v}{\partial x_i}$ is a $SBV_{loc}(\Omega)$ function (see section 2 for  further details).  
If ${J}_{\frac{\partial v}{\partial x_i}}$ denotes the 
jump set of $\frac{\partial v}{\partial x_i}$, $i=1,\dots,N$,  
one could try to minimize the functional
$$
\mathcal{J}_t(v):= \mathcal{H}^t\left(\bigcup_{i=1}^N {J}_{\frac{\partial v}{\partial x_i}}\right)
$$ 
over
$$
\mathcal{E}(\Omega)=\left\{v \in \mathfrak{S}(\Omega): \frac{\partial v}{\partial x_i}\in SBV_{loc}(\Omega), i=1,\dots,N\right\},
$$ for some $t\geq N-1$.
As shown in \cite{Champion-Croce},  for $t=N-1$, this
variational problem is not well-posed, since $\mathcal{J}_{N-1}(v)$ can
be infinite for every $v \in \mathcal{E}(\Omega)$, in general. Indeed
the jump set of $Dv$ could have a fractal behaviour near the boundary of $\Omega$. 
From the other hand,
$\mathcal{H}^t\left({J}_{\frac{\partial v}{\partial x_i}}\right)=0
$ for every $i=1,\dots,N$, for every function in $\mathcal{E}(\Omega)$ 
and for every $t>N-1$ (see Section \ref{Main result}).
The above negative result on 
the variational problem
$$
\inf\left\{\mathcal{H}^{N-1}\left(\bigcup_{i=1}^N{J}_{\frac{\partial v}{\partial x_i}}\right), v\in \mathcal{E}(\Omega)\right\}
$$
leads the authors of  \cite{Champion-Croce} to consider a weighted $\mathcal{H}^{N-1}$ measure of the jumps. Using a given increasing 
sequence $\Omega_n\subset \Omega$ of polyhedra
whose boundary is composed of a finite number of hyperplanes with
normals lying in $E$, 
they define a $C_0$ function $h:\Omega \to \R^+$ and  minimize
$$
v\to \sum_{i=1}^N\int_{J_{\frac{\partial v}{\partial x_i}}} h(x)d\mathcal{H}^{N-1}(x)
$$
over $\mathcal{E}(\Omega)$. 
It is obvious that this selection criterion depends on the 
definition of the function $h$ and in turn, $h$  depends on the
particular sequence of sets $\Omega_n$. 
At the same time it is clear that the candidate functional should
depend on some geometric properties of the domain $\Omega$. 

Our aim in this paper is to find an "optimal" weighted measure of the jump set of $Dv$. 
We propose a family of weights which depends only on the distance from the boundary 
of $\Omega$ and then minimize the corresponding variational problem.

We  assume that $\Omega\subset \R^2$ is a {\it compatible} domain, according to Definition \ref{def:01}. Observe that any convex, open, bounded, Lipschitz subset of $\R^2$ such that $\partial \Omega$ is composed of a finite number of $C^1$ curves is compatible. A polygon is a compatible domain, as well.
This hypothesis is motivated in section 3 through some enlightening examples. Our main result is the following
\begin{theorem}\label{enunciato_intro}
Let $\Omega$ be a compatible domain of $\R^2$.
Let $H: \R^+\to\R^+$ be a continuous, increasing function such that 
$$
\int_0^1 \frac{H(t)}{t}dt<+\infty\,.
$$
Let 
$$
\mathcal{F}(v)=\sum_{i=1}^2\int_{\Omega}H(d_1(x,\partial \Omega))d\left|D{\frac{\partial v}{\partial x_i}}\right|\,.
$$
Then the variational problem
\begin{equation}\label{variational_problem_intro}
\inf\{\mathcal{F}(v), v\in \mathcal{E}(\Omega)\}
\end{equation}
has  a solution.
\end{theorem}
We point out that the most difficult part of the proof of Theorem \ref{enunciato_intro} consists in showing 
that the variational problem (\ref{variational_problem_intro}) is well-posed, that is,
defining a function $v\in \mathcal{E}(\Omega)$ such that $\mathcal{F}(v)$ is finite. 
This will be done in Proposition \ref{functional_finite}.
In section 3 we sketch the proof of Proposition \ref{functional_finite} in the special case $\Omega=(0,1)\times (0,1)$ where the computations are much easier than in the general case; 
we motivate there the hypothesis on $H$.
After proving that $\mathcal{F}$ is finite for at least one function in $\mathcal{E}(\Omega)$, we show the existence of a minimizer  using the direct methods of the calculus of variations, 
as in \cite{Champion-Croce}.
We finally remark that the generalization of our result to higher dimensions
is not immediate, because in that case Proposition \ref{functional_finite} 
is much more complicated. This is why we state our main result in $\R^2$. However, we plan to address the problem in higher dimensions in a future work. 
 
\medskip
\noindent{\bf Acknowledgements} We would like to thank T. Champion for some useful discussions. 
Part of this work was completed during a visit of the first author to
Seconda Universit\`a di Napoli and of the second author to the
Hokkaido University whose hospitality is gratefully  
acknowledged.

\section{Notations and preliminaries}
This section is divided into two parts. In the first one we will fix
the hypotheses on $\Omega$ and the related notations. The second one
is devoted to some preliminary results on functions of bounded
variation.

Throughout, given a real number $t$ we will use the notation $[t]$ do
denote its integer part. Given a continuous function $f: [a,b]\to \R$, $f\in C^1((a,b))$ with $f'(t)<0$ for every $t\in (a,b)$, we  set
\[
T_{f}:=\{(s,t)\in \R^2: a\leq s \leq b, f(b)\leq t\leq f(s)\}.
\]
We will call $T_f$ a \emph{triangular domain}. The class of all  triangular
domains will be denoted by $\mathcal{T}$. We will write $T$ instead of
$T_f$ if the function $f$ will be clear by
the context.

In the sequel a special role is played by the family of lines in the plane with slope $\pm 1$. 
We will denote by $\ell_{\pm}$ the line with slope $\pm 1$,  passing through $(0,0)$. Let 
$\Lin$ be the set of lines parallel to $\ell_{+}$ or $\ell_{-}$.
We define
$$\Seg:=\left\{ [x,y]  \;: [x,y]\subset \ell \;,\;\;\ell \in \Lin  \;  \right\}\,,
$$
where $[x,y]$ stands for the segment
joining the two points $x=(x_{1},x_{2})$, $y=(y_1,y_2)$.  
    We will denote by $\mathcal{P}$ the class of bounded Lipschitz domains whose 
    boundary can be written as an at most countable union of segments 
    lying in $\Seg$;  $\Pfin \subset \mathcal{P}$ will be the class of domains 
    whose boundary can be written as a finite union of segments in $\Seg$;  
    finally,  $\Pinf=\mathcal{P} \setminus \Pfin$. 
 For a given set $A\subset \R^{2}$, with $\interior{A}$ we denote the
 relative interior of $A$ with respect to the topology induced by the euclidean metric on $\R^{2}$.
Let 
$$
\Rot = \left(
  \begin{array}{cr}
    \frac{1}{\sqrt{2}} & -\frac{1}{\sqrt{2}} \\ \frac{1}{\sqrt{2}} & \frac{1}{\sqrt{2}}
  \end{array}
\right)
\,\,\,\textnormal{and}\,\,\,
\Ref = \left(
  \begin{array}{rc}
   -1 & 0 \\ 0 & 1
  \end{array}
\right)
$$
be the $\frac{\pi}{4}$ counterclockwise  rotation and the reflexion with 
respect to the vertical axis respectively.
We define 
$$
\mathfrak{T}=\{\Rot^{2k+1}(T), \Ref(\Rot^{2k+1}(T)), k=0, 1, 2, 3, T
\in \mathcal{T}\}\,.
$$ 
 Let us consider the set  $E:=\{ \nu\in S^{1}\,:\, |\nu_{1}|=|\nu_{2}| \}$. 
    For a given rectifiable curve of class $C^{1}$, $\gamma \subset \R^{2}$, 
    we denote by $N:\gamma \mapsto S^{1}$ the Gauss map.
    \begin{definition}\label{def:01}
   We say that $\gamma$ 
    is an \emph{admissible boundary curve} if 
    the set $N^{-1}(E)\subset \gamma$ has a finite number of connected components. 
    We will denote by $\Gamma$ the class of admissible boundary curves. 
   A bounded connected Lipschitz set $\Omega\subset \R^{2}$ will be 
   called \emph{compatible domain} if its boundary $\partial \Omega$ 
   is the union of a finite number of admissible boundary curves.
     \end{definition}
     We recall that an open bounded set $\Omega\subset \R^2$ is Lipschitz 
     if for every $p \in \partial \Omega$ there exist a radius $r>0$ and a map $h_{p}: B_r(p)\to B_1(0)$ such that
     $h_p$ is a bijection, $h_p, h_p^{-1}$ are Lipschitz functions, $h_p(\partial \Omega \cap B_r(p))=Q_0$ and
$h_p(\Omega \cap B_r(p))=Q_+$, where $B_r(p)=\{x \in \R^2: ||x-p||\leq r\}$, $Q_0=\{x=(x_1,x_2) \in B_1(0): x_2=0\}$ and
$Q_+=\{x=(x_1,x_2) \in B_1(0): x_2>0\}$. We remark that if $\Omega$ is Lipschitz, 
then $\mathcal{H}^1(\partial \Omega)$ is finite. 

We observe that any domain in $\Pfin$ is compatible, since any segment 
$\sigma \in \Seg$ is an admissible boundary curve, being 
$N^{-1}(E)= \interior{\sigma}$ connected.
Any convex, open, bounded, Lipschitz subset of $\R^2$  such that $\partial \Omega$ 
is composed of a finite number of $C^1$ curves is compatible. 
A polygon is also a compatible domain.
\begin{remark}\label{covering_net}
If $\Omega$ is a compatible domain, then it can be covered by  a
polygonal set in $\Pfin$ and a finite number of 
domains in $\mathfrak{T}$ with mutually disjoint interior, i.e.,
$$
\overline{\Omega}=\overline{P}\cup \bigcup_{j=1}^N \mathcal{W}_j(T_j),
$$
where $P\in \mathcal{P}^f$, $N\in \N$, $T_j \in \mathcal{T}$
and $\mathcal{W}_j$ is a linear transformation of the form 
$\Rot^{2k_j+1}$ or $\Ref(\Rot^{2k_j+1})$, with  $k_j\in\{0, 1, 2, 3\}$.
\end{remark}

As already explained in the introduction, in this article we minimize a weighted measure 
of the  discontinuity set of the gradient of  the solutions to problem
(\ref{SEE}).

To do that, we will use the definition and some results on the spaces
$BV(\Omega)$ and $SBV(\Omega)$ of functions of bounded variation that
we recall here. For simplicity we will restrict ourselves to the case
where $\Omega$ is a subset of $\R^2$. We refer to \cite{A-M, ambrosio-fusco-pallara, evans-gariepy,federer} for further details.

\begin{definition}[$BV, SBV$ function]
A $BV(\Omega)$ function is an $L^1(\Omega)$ function such that its distributional gradient is a finite Radon measure.
A $SBV(\Omega)$ function is a $BV(\Omega)$ function such that its gradient can be decomposed in
the following way:
\begin{eqnarray*}
Dw & = & D^a w + D^j w \\
& = & \nabla w \, \mathcal{L}^2+(w^+-w^-)\nu_w\,\mathcal{H}^{1}\lfloor \Saut_w
\,\, = \,\,
\nabla w\mathcal{L}^2+[w]\nu_w\mathcal{H}^{1}\lfloor \Saut_w
\end{eqnarray*}
where $D^a w$ is absolutely continuous with respect to
the Lebesgue measure $\Leb^2$ in $\R^2$ with density $\nabla w$, $D^j w = [w]\nu_w\mathcal{H}^{1}\lfloor \Saut_w$
is the {\it jump} part of $Dw$,
$\Haus^{1}$ is the one-dimensional Hausdorff measure,
$w^+$ and $w^-$ denote the upper and lower approximate limits of $w$, $J_w$ the jump set of $w$ and $\nu_w$ its generalized normal. 
\end{definition}    

We shall handle the solutions $v$ of problem (\ref{SEE}) belonging to 
$$
\mathcal{E}(\Omega)=\left\{v\in \mathfrak{S}(\Omega): v_{x_i} \in SBV_{loc}(\Omega), i=1,2\right\}
$$  where  $v_{x_i}$ stands for $\frac{\partial v}{\partial x_i}, i=1,2$.
Therefore
\begin{equation}\label{DVrepresentation}
D\left(v_{x_i}\right)
=2 \, \nu_{v_{x_i}} \,\Haus^{1} \lfloor \Saut_{v_{x_i}}
\quad on \,\, \omega\quad \forall\, i=1,2\,
\end{equation}
for any open subset $\omega \subset \overline{\omega}\subset \om$.

We recall moreover the following result:
\begin{lemma}\label{federer}
If $P: \R^2\to \R$ is the projection $(x_1,x_2)\to x_2$,
and $E$ is a measurable set of $\R^2$ then
$$
\int_{\R}\mathcal{H}^0(E\cap P^{-1}\{y\})dy\leq \mathcal{H}^{1}(E).
$$
\end{lemma}
\begin{proof}
It is sufficient to apply Theorem 2.10.25 of \cite{federer}
with $f=P$, $X=\R^2$, $Y=\R$, $k=0$ and $m=1$.
\end{proof}

We end this section with some results on compactness and 
lower semi-continuity in the space of  functions of bounded variation.
\begin{definition}[$BV$ norm]
The $BV$ norm of a $BV(\Omega)$ function $w$ is defined by
$$
\norm{w}_{BV(\Omega)}=\norm{w}_{L^1(\Omega)}+|Dw|(\Omega)\,.
$$
\end{definition}

\begin{definition}[weak* convergence in $BV$]
Let $(u_n)_n, u \in BV(\Omega)$. We say that $(u_n)_n$ weakly* converges to $u$ in $BV(\Omega)$ if
$u_n \to u$ in $L^1(\Omega)$ and the measures $Du_n$ weakly* converge to the measure $Du$ in
$\mathcal{M}(\Omega, \R^2)$, that is,
$$
\lim\limits_{n \to \infty}\int_{\Omega}\varphi dD u_n=
\int_{\Omega}\varphi dD u \qquad \forall\,\varphi \in C_0(\Omega)\,.
$$
\end{definition}


\begin{theorem}[compactness for $SBV$ functions]\label{compactnessSBV}
Let $\Omega$ be a bounded open subset of $\R^2$ with $\mathcal{H}^{1}(\partial \Omega)<+\infty$.
Let $(u_n)_n$ be a sequence of functions in $SBV(\Omega)$ and assume that:
\\
$i)$ the functions $u_n$ are uniformly bounded in $BV(\Omega)$;
\\
$ii)$ the gradients $\nabla u_n$ are equi-integrable;
\\
$iii)$ there exists a function $f:[0,\infty)\to [0,\infty]$ such that
$f(t)/t \to \infty$ as $t \to 0^+$ and
$$
 \int_{\Saut_{u_n}}f([u_n])d\mathcal{H}^{1}
\leq C<\infty\quad \forall n \in \N.
$$
Then there exists a subsequence $(u_{n_k})_k$ and a function $u \in SBV(\Omega)$ such that
$u_{n_k}\to u$ weakly*  in $BV(\Omega)$,
with the Lebesgue and jump parts of the
derivatives converging separately, i.e., $D_a u_{n_k}\to D_a u$ and $D_j u_{n_k}\to D_j u$
weakly* in $\mathcal{M}(\om,\R^2)$.
\end{theorem}

\begin{theorem}[semicontinuity in $BV$]\label{semicontinuityBV}
Let $\Omega$ be a bounded open subset of $\R^2$.
Let $(u_n)_n$ be a sequence of functions in $BV(\Omega)$ such that $u_n \to u$ 
weakly* in $BV(\Omega).$ Then
$$
\int_\om f(x) d|D_j u|(x) \leq \liminf_{n \to \infty}\int_\om f(x) d|D_j u_{n}|(x)
$$
for any non-negative continuous function $f: \om \to [0,+\infty[\,$.
\end{theorem}


\section{Statement of the main result and remarks}\label{Main result} 
As already explained in  the introduction, our strategy to select some
particular or special solutions to \eqref{SEE} is based on a
variational method, in other words we search for minimizers of a suitable 
and meaningful functional defined on  $\mathcal{E}(\Omega)$. 
We will motivate the choice of the functional and the conditions imposed on the domain $\Omega$ with the 
aid of some enlightening examples.
Our main result is the following
\begin{theorem}\label{main_thm}
Let $\Omega$ be a compatible domain and $H: \R^+\to \R^+$  be an increasing, continuous function such that
\begin{equation}\label{hyp_H_crescita}
\int_0^1\frac{H(t)}{t}\,dt<\infty\,.
\end{equation}
Let
\[
\mathcal{F}(v)=\sum_{i=1}^2\int\limits_{\Omega}H(d_1(x,\partial \Omega))d(|D\, v_{x_i}|)(x)\,.
\]
Then the variational problem 
\begin{equation}\label{variational_problem}
\inf\{\mathcal{F}(v)\,,v\in \mathcal{E}(\Omega)\}
\end{equation}
has a solution.
\end{theorem}

One could ask why we do not minimize
the simpler functionals
$$\dys \mathcal{H}^t\left(J_{v_{x_1}}
\cup J_{v_{x_2}}\right)\,,\,\, t\geq 1$$
over $\mathcal{E}(\Omega)$.
To give the answer we will distinguish the cases $t=1$ and $t>1$.
The first case was partially studied in \cite{Champion-Croce}, where
it was proved the following theorem for the  functional
$$
\Fold{\Omega}(v):=\sum_{i=1}^2\int_{\Omega}d(|D v_{x_i}|)(x)\,.
$$ 
\begin{theorem}\label{cha-cro}
Let $P \in \Pfin$.
Then the variational problem 
$$
\inf\{\Fold{P}(v), v\in \mathcal{E}(P)\}
$$
has a solution.
\end{theorem}
If $\Omega$ does not belong to $\Pfin$, the minimization of 
$
\Fold{\Omega}$ may  be not well-posed. 
Indeed we recall that in \cite{Champion-Croce} it was proved that
if $\Omega=(0,1)\times (0,1)$ then
$\Fold{\Omega}(v)$ is infinite for every $v\in \mathcal{E}(\Omega)$.
In Examples \ref{scala_infinita_cattiva} and \ref{scala-buona} we consider the case where $\Omega$ is a domain in $\mathcal{P}^{\infty}$.

\begin{example}\label{scala_infinita_cattiva}
We define a set $\Omega \in \mathcal{P}^{\infty}$
such that for every $v\in \mathcal{E}(\Omega)$  $\Fold{\Omega}(v)=\infty$.
We start fixing some notations. 
For a function $f:[a,b]\to \R^+$ we will denote by
$S_f$ the set
$$
S_f=\{(x_1,x_2)\in \R^2: a\leq x_1\leq b, 0\leq x_2\leq f(x_1)\}.
$$
For a given $N\in \N$, with the symbol $g_{(a,b)}^N$ we will denote the step function
$$
g_{(a,b)}^N=\sum_{j=0}^{N-1}{(b-a)}\frac{N-j}{N}\chi_{\left[a+{j}\frac{b-a}{N},{a+(j+1)}\frac{b-a}{N}\right)}\,.
$$ 
Let us consider the sequence of real numbers $t_{0}=0$, $t_{n}=\sum_{i=1}^{n}\frac{1}{2^{i}}=\frac{2^{n}-1}{2^{n}}$, 
and define  for $n\in \N$ the functions
\[
h_{n}=\frac{1}{2^{n}}\chi_{\left[ t_{n-1},t_{n} \right)} \;,\;\;\; g_{n}=g^{N_{n}}_{(t_{n-1},t_{n})}\;,
\]
\[
h=\sum_{n=1}^{+\infty}h_{n}\;,\;\;\; f=\sum_{n=1}^{+\infty}(h_{n}+g_{n})\;,
\]
where $N_{n}$ will be chosen later. Let $\Omega$ be
$\Rot(\interior{S_{f}})$. We can easily see that, if we set $C_{n}=\Rot(S_{g_{n}+\frac{1}{2^{n}}})$, 
then  $\interior{C_{i}}\cap \interior{C_{j}}=\emptyset$ and 
\[
\Rot(S_{f})=\Rot(S_{h})\cup \bigcup_{n=1}^{\infty}C_{n}\,.
\]
Thus the claim will be proved if we show that for every $n\in \N$
$$
\mathcal{H}^1\left(J_{v_{x_1}}\cap \interior{C_n}\right)+
\mathcal{H}^1\left(J_{v_{x_2}}\cap \interior{C_n}\right)\geq 1\,.
$$

Let $\sigma_n$ be the third middle part of the segment
$\big[\Rot \big( (t_{n-1},f(t_{n-1}))\big),\Rot\big ( (t_n,f(t_n))\big)\big]$; then
$\mathcal{H}^1(\sigma_n)=\frac{\sqrt{2}}{3\,2^n}$. For any $0<t<\frac{\mathcal{H}^1(\sigma_n)}{4}$ consider
the segment $\sigma_n^t=\tau_{-t}(\sigma_n)$, where $\tau_{w}, w\in \R,$ denotes 
the translation in the direction $(0,w)$. Clearly we have
$\sigma_n^t\subset \interior{C}_n$.
Now, let $r_t(s)$ be a parametrization for
$\sigma^t_n$.  If  $w_t(s)$ denotes the restriction of a
function $v\in\mathcal{E}(\Omega)$ 
to the segment $\sigma^t_n$, 
we have that 
$$
|w_t(s)| \leq d_{1}(r_t(s),\partial \Omega) = d_{1}(r_t(s),\partial C_n)\leq t+h_n\,,\,\,\,\,\,h_n=\frac{\sqrt{2}}{2^{n+1} N_n}\,.
$$ 
This implies that $w_t'$ has at least $\left[\frac{\mathcal{H}^1(\sigma_n)}{2(t+h_n)}\right]$ jumps.
Therefore, by Lemma \ref{federer}
\[
\begin{split}
\mathcal{H}^1\left(J_{v_{x_1}}\cap \interior{C}_n\right) &  \geq 
\int_{0}^{\frac{\mathcal{H}^1(\sigma_n)}{4}} \left[\frac{\mathcal{H}^1(\sigma_n)}{2(t+h_n)}\right]dt 
\\ & \geq  \int_{0}^{\frac{\mathcal{H}^1(\sigma_n)}{4}} \left(\frac{\mathcal{H}^1(\sigma_n)}{2(t+h_n)}-1\right)dt
\\ & =
\frac{\mathcal{H}^1(\sigma_n)}{2}\ln\left(\frac{\mathcal{H}^1(\sigma_n)}{4h_n}+1\right)-\frac{\mathcal{H}^1(\sigma_n)}{4}
\\ 
 &  = \frac{\sqrt{2}}{3\,2^{n+1}}\ln\left(\frac{\sqrt{2}}{12}N_n+1\right)-\frac{\sqrt{2}}{3\cdot 2^{n+2}}\,.
\end{split}
\]
A suitable choice of $N_n$ implies the claim.
\end{example}



\begin{example} \label{scala-buona}
We are going to define a domain $\Omega \in \mathcal{P}^{\infty}$ 
and a function $v\in \mathcal{E}(\Omega)$ such that
$\Fold{\Omega}(v)$ is finite. 
Let $t_n$ be defined as in the previous example. For any $n\in \N$ choose
$h_n< \frac{1}{2^n}$ and consider the domain $\Omega=\Rot(\interior{S_{f}})$ with
$f$ defined by 
\[
f=\sum_{n=1}^\infty h_n \chi_{(t_{n-1},t_n)}.
\]
Let $R_n=[t_{n-1},t_n]\times [0,h_n]$ and define the function
$v(x)=d_1(x,\partial R_n)$ if $x\in R_n$. Clearly $v$ is a solution
to problem  \eqref{SEE}; moreover $\Rot(S_f)=\cup_{n=1}^\infty \Rot(R_n)$ and
\[
\Hun(J_{v_{x_i}}\cap \Rot(R_n)) \leq 4 \left(
    h_n+\frac{1}{2^n}\right) \leq \frac{8}{2^n}. 
\]
It follows the following estimate: 
\[
 \begin{split}
 \Fold{\Omega}(v)  & \leq \Hun(J_{v_{x_1}})+\Hun(J_{v_{x_2}}) \\
    & \leq \sum_{n=1}^\infty  \left[\Hun(J_{v_{x_1}}\cap \Rot(R_n))
    + \Hun(J_{v_{x_2}}\cap \Rot(R_n)\right] \\
    & \leq  C
      \sum_{n=1}^\infty \frac{1}{2^n} < \infty\,,
\end{split}
\]
where $C$ denotes a positive constant independent of $n$.
\end{example}

\bigskip
We now pass to the case $t>1$. The reason why we do not use the functionals $\dys \mathcal{H}^t\left(J_{v_{x_1}}
\cup J_{v_{x_2}}\right)$ for some $t>1$, to select a solution to problem (\ref{SEE}) is clear: for every $v\in \mathcal{E}(\Omega)$, 
$$\dys \mathcal{H}^t \left(J_{v_{x_1}}
\cup J_{v_{x_2}}\right)=0\,,\,\,\,\forall\,t>1.
$$
Indeed,
for a given $v \in \mathcal{E}(\Omega)$
$$
\mathcal{H}^1\left(J_{v_{x_i}}\cap \omega\right)<\infty\,,
$$
for every 
$\omega\subset \overline{\omega}\subset \Omega$  and for  $i=1,2$. 
Recall now that Hausdorff measures have the property that
if $E$ is any measurable set such that $\mathcal{H}^r(E)<+\infty$, then $\mathcal{H}^{r+\varepsilon}(E)=0$ for every $\varepsilon>0$ (see \cite{evans-gariepy}).
This implies that
$\mathcal{H}^t\left(J_{v_{x_i}}\cap \omega\right)=0$ for every $t>1$. By definition of measure 
$\mathcal{H}^t \left(J_{v_{x_i}}\right)=0$ for every $t>1$. We remark
that this argument can be applied also if we work in dimension greater 
than 2. 

The previous observations suggest that, if one wants to isolate the most "regular" functions in $\mathcal{E}(\Omega)$, he is in some sense obliged  to use a weighted $\mathcal{H}^1$ measure.
In this article we define a weight depending only on the distance to the boundary.

In the following example we are going to motivate hypothesis
 \eqref{hyp_H_crescita} on $H$ 
and to illustrate
the ideas behind the proof of Theorem \ref{main_thm}, in the  case  where $\Omega$ is a  square. 
Due to the simple
geometry of the domain, we can perform several steps of the
proof avoiding almost all the technical difficulties that we need to
deal with in the general case and that are addressed in the next
section.
\begin{example}
 \label{ex-square}
Let $\Omega =(-1,1)\times(-1,1)$.
We are going to define a "reasonable" function $v \in \E(\Omega)$ and
impose that $\F(v)$ is finite.
Since the weight $H$ in the definition of $\F$ depends only on the
distance from the boundary,
our aim is to  investigate how fast the discontinuity
of  $\frac{\partial v}{\partial x_i},\, i=1,2$ develops near the
boundary. With this
information we will be able to define an appropriate weight $H$.

We define $v$ as follows.
Let $R$ be the triangle
with vertices in $(-1,1)$, $(1,1)$ and $(0,0)$. Then $\Omega$
is union of the counter-clockwise $\pm \frac{\pi}{2}$ and
$\pi$-rotations of $R$.
Therefore it is sufficient to define $v$ in $R$.
We fill $R$ with squares belonging to
$\Pfin$;  consequently they can be identified
once we know the position of the upper vertex and the length of the
diagonal. Let $\Omega_0$ be the square in $\Pfin$ with upper vertex in
$\left(0,1\right)$ and length of the diagonal equal to
$1$. For $n\in \N$ and $i=1,\dots,2^{n-1}$ let $Q_n^i$ be the square in $\Pfin$
with upper vertex in $
\left(
\frac{2i-1}{2^n},1
\right)$
and length of the diagonal equal to $\frac{1}{2^{n}}$. We set
$Q_n^{-i}= \Ref(Q_n^i)$. With this notation we can define the
following covering of $R$ made up by squares with
mutually disjoint interior:
\[
\overline{R}=\bigcup_{n=0}^\infty  \Omega_n \;\;\;,\;\;\;\;
\Omega_n:=\bigcup_{i=1}^{2^{n-1}} Q_n^i\cup Q_n^{-i}.
\]
Observe that the north corner of each square belongs to $\gamma=[(-1,1),(1,1)]$.
Now we define the solution $v$ as
\[
\begin{array}{lcl}
\displaystyle v_0(0)= d_1(x, \partial \Omega_0)\chi_{\Omega_0} & ,&
v_n^i(x)=d_1(x, \partial Q_n^i)\chi_{Q_n^i}\;, \\ \\
v_n^{-i}(x)=d_1(x, \partial Q_n^{-i})\chi_{Q_n^{-i}} & , &
\displaystyle v_n(x)=\sum_{i=1}^{2^{n-1}} [v_n^i(x) + v_n^{-i}(x)] \,; \\
\end{array}
\]
\[
v(x)= \sum_{n=0}^\infty v_n(x).
\]
It is clear from the definition that $v\in \mathcal{E}(\Omega)$ and
that the distributional gradient  of $Dv$ is supported on the
sides and the diagonals of the squares we used to fill $R$.
In order to estimate them, let us fix some notations.
Let 
\[
R_m=\left\{(x_1,x_2)\in R: \frac{1}{m+1}<x_2\leq \frac 1m\right\}.
\]
We denote by:
\begin{itemize}
\item
$\mathcal{S}$ the union of the sides of the squares $Q_n^{\pm i}$,
\item
$\mathcal{D}_v$ the union of the vertical diagonals of the squares
$Q_n^{\pm i}$,
\item
$\mathcal{D}_h$ the union of the horizontal diagonals of the squares
$Q_n^{\pm i}$,
\end{itemize}
for  $n\in \N$ and $i=0,\dots,2^{n-1}$.
Therefore $\mathcal{F}(v)$ is finite if the three following quantities
$$
\int_{\mathcal{D}_h}H(d_1(x,\gamma))d\mathcal{H}^1\,,\,\,\,\,\,\,\,\,\,\,
\int_{\mathcal{D}_v}H(d_1(x,\gamma))d\mathcal{H}^1\,,\,\,\,\,\,\,\,\,\,\,
\int_{\mathcal{S}}H(d_1(x,\gamma))d\mathcal{H}^1
$$
are finite.
The first quantity is very simple to estimate, if one remarks that the
horizontal diagonals appear only
at heigth $\frac{1}{2^n}, n\in \N$, and their total length is $1$ for
every $n\in \N$.
Therefore
$$
\int_{\mathcal{D}_h}H(d_1(x,\gamma))d\mathcal{H}^1\leq
\sum_{n=1}^{\infty}H\left(\frac{1}{2^n}\right)\,.
$$
Using the Cauchy's condensation criterion, the last series is finite
if and only if
\begin{equation}\label{Cauchy_example}
\sum_{n=1}^{\infty}H\left(\frac{1}{n}\right)\frac 1n<\infty\,.
\end{equation}
To estimate the sides and the vertical diagonals we will use a
different strategy.
We observe that the number $N_m$ of squares which intersect $R_m$
is bounded by $m+1$. Indeed
if $x_0$ is the north corner of a square intersecting $R_m$, the
distance from $x_0$ to the next north corner
of a square intersecting $R_m$ is at least $\frac{2}{m+1}$. Therefore
$$N_m=\frac{\mathcal{H}^1(\gamma)}{\frac{2}{m+1}}=m+1\,.$$
Now, let $Q$ be a square intersecting $R_m$. The
the $\mathcal{H}^1$ measure of the intersection of the vertical
diagonal of $Q$ with $R_m$ is less or equal to $\frac
1m-\frac{1}{m+1}=\frac{1}{m(m+1)}$ and
the $\mathcal{H}^1$ measure of the intersection of one of the sides of
$Q$ with $R_m$ is bounded by $\sqrt{2}\left(\frac
1m-\frac{1}{m+1}\right)=\frac{\sqrt{2}}{m(m+1)}$.
Therefore
$$
\int_{\mathcal{D}_v}H(d_1(x,\gamma))d\mathcal{H}^1\leq
\sum_{m=1}^{\infty}H\left(\frac 1m\right)N_m\frac{1}{m(m+1)}\leq 
\sum_{m=1}^{\infty}H\left(\frac 1m\right)\frac 1m
$$
and
$$
\int_{\mathcal{S}}H(d_1(x,\gamma))d\mathcal{H}^1\leq
2\sum_{m=1}^{\infty}H\left(\frac 1m\right)N_m\frac{\sqrt{2}}{m(m+1)}\leq
2\sqrt{2}\sum_{m=1}^{\infty}H\left(\frac 1m\right)\frac 1m\,.
$$
Hence we find the same condition on $H$ as in (\ref{Cauchy_example}):
$$\sum_{m=1}^{\infty}H\left(\frac{1}{m}\right)\frac 1m<\infty\,.$$
\end{example}

\section{Proof of the main result}
In this section we are going to prove Theorem \ref{main_thm}.
The proof can be divided into two parts: in the first we show that the variational problem
(\ref{variational_problem}) is well-posed (see Proposition
\ref{functional_finite}); in the second one we prove that there exists
a minimizer of the functional $\mathcal{F}$ in $\mathcal{E}(\Omega)$.

We are going to concentrate on the first step, that is, we are going to construct a function $v\in \mathcal{E}(\Omega)$ such that
$\mathcal{F}(v)$ is finite.
For this purpose, recalling that $\Omega$ is a compatible domain, 
according to Remark \ref{covering_net}, we have
\begin{equation}\label{decomposition}
\overline{\Omega}=\overline{P}\cup\bigcup_{j=1}^N \mathcal{W}_j(T_j)
\end{equation}
where $P\in \mathcal{P}^f$, $T_j \in \mathcal{T}$ and
$\mathcal{W}_j(T_j)=\Rot^{2k_j+1} (T_j)$ or
$\Ref(\Rot^{2k_j+1} (T_{j}))$ for some $k_j \in
\{0,1,2,3\}$. 

With the aim of setting $v$ in each domain $\mathcal{W}_j(T_j)$ we will define a special countable covering of the
interior of a triangular domain made up by squares. We start by introducing three operators defined
in $\mathcal{T}$. For a given 
$$T=\{(x_1,x_2): a \leq x_1 \leq b, h(b)
\leq x_2 \leq h(x_1)\}\in \mathcal{T}\,,$$ 
 let $x_1^0$ be such that
$h(x_1^0)=x_1^0+h(b)-a$ and define
\[
\begin{array}{lcl}
 \displaystyle  q(T) & := &  \displaystyle  (a,x_1^0)\times (h(b),h(b)+x_1^0-a); \\
  u(T) & := & \displaystyle  \{(x_1,x_2)\in T: a<x_1<x_1^0, h(b)+x_1^0<x_2<h(x_1)\}; \\
  r(T) & := &  \displaystyle  \{(x_1,x_2)\in T: a+x_1^0<x_1<b, h(b)<x_2<h(x_1)\}.
\end{array}
\]
We explicitly observe that $u$ and $r$ have values in $\mathcal{T}$
while $q$ maps  any triangular domain to a square in $\Rot(\Pfin)$.


\begin{definition}[of the covering of $T$]\label{definition_covering}
Let for $m \in \N$
\[
S_m:=\{ \sigma=(\alpha_1,\dots,\alpha_m) \;:\;\alpha_i\in \{u,r\} \}\,,
\]
be the set of all the $m$-permutations of the two letters $u$ and
$r$. For $\sigma\in S_m$, using the notation 
\[
\sigma(T)=\alpha_1 \circ \alpha_2 \circ \cdots \circ \alpha_m(T),
\]
we set
\[
Q_T^{m,\sigma}= q(\sigma(T))\;\;;\;\;\;\sigma \in S_m.
\]
We finally define the following family of squares contained in $T$:
\[
\mathcal{Q}(T):=\{Q_T^{m,\sigma} \;:\; m\in \N\cup \{0\} \;,\;\;\sigma\in S_m\}.
\]  
\end{definition}
\begin{remark}
We remark that $\mathcal{Q}(T)$ is a covering of
$\interior{T}$ composed of squares with mutually disjoint interiors
belonging to $\Rot(\Pfin)$. It  may be useful to think of $\mathcal{Q}(T)$ as being constructed in steps, starting from $m=1$ and adding at 
step $m$ the squares $Q_T^{m,\sigma}$ with $\sigma\in S_m$. Since the
cardinality of $S_m$, $\sharp (S_m)$, is equal to $2^m$, we add $2^m$ squares at the $m$-th
step. Therefore the first steps of the construction are:

\vspace{0.3cm}
\begin{tabular}{lcll}
\vspace{0.2cm}
{\it Step 0.} & We start with &   $Q^0_T=q(T)$ ; & \\ 
\vspace{0.2cm}
{\it Step 1.} & we add & $Q_T^{1,u}=q(u(T))$ \;and &  $Q_T^{1,r}=q(r(T))$ ; \\ \vspace{0.2cm}
{\it Step 2.} &  we add &  $Q_T^{2,(u,u)}=q(u(u(T)))$, &
$Q_T^{2,(u,r)}=q(u(r(T)))$, \\
\vspace{0.2cm}
 & &  $Q_T^{2,(r,u)}=q(r(u(T)))$, & $Q_T^{2,(r,r)}=q(r(r(T)))$ \\
\vspace{0.2cm}
{\it Step 3.} &  we add & $Q_T^{3,(u,u,u)}=q(u(u(u(T))))$, &
$Q_T^{3,(u,u,r)}=q(u(u(r(T))))$, \\
 &  & $\cdots$ & $\cdots$ 
\end{tabular}
In the sequel, when $\sharp (S_m)$ will play a role,
we will use for $Q_T^{m,\sigma}$ the notation $Q_T^{m,k}$,
$k\in\{1,\dots,2^m\}$. 
\end{remark}

%
%

We are now in a position to define the candidate function $v:\Omega
\to \R^+ $ such that $\mathcal{F}(v)$ is finite, which is the
natural generalization of the construction made in Example \ref{ex-square}.
\begin{definition}[of the candidate function]\label{defn_function}
Let $w$ be a minimizer of the functional $\Fold{P}$. For a fixed $j\in \{1, \dots,
N\}$, let
\[
v_m^j(x)= \sum_{\sigma \in S_m}d_1\left(x,\partial \left(
    \mathcal{W}_j(Q_{T_j}^{m,\sigma})\right)\right)\chi_{Q_{T_j}^{m,\sigma}}(x)\, ;
\]
\[
v_j (x)= \sum_{m\in\N} v_m^j(x).
\]
We define
\[
v(x)=w(x)\chi_{P}(x)+\sum_{j=1}^Nv_j(x).
\]
\end{definition}

In the following we will prove that problem (\ref{variational_problem}) is well-posed.
We remark that in $\Omega\setminus P$, $Dv_{x_1}$ and $Dv_{x_2}$ are
supported on the sides of each square
$\mathcal{W}_j(Q^{m,\sigma}_{T_j})$ and on its vertical and horizontal
diagonals respectively. For this reason, we denote by
$\mathcal{D}_T^{+}$ and $\mathcal{D}_T^{-}$ the union of all the diagonals
of the squares in  $\mathcal{Q}(T)$ parallel to $\ell_{+}$ and
$\ell_{-}$ respectively, and  by $\mathcal{S}_T$ the union of all the sides of the squares in $\mathcal{Q}(T)$. 

\begin{proposition}\label{functional_finite}
The function $v\in \mathcal{E}(\Omega)$ defined in \ref{defn_function} satisfies $\mathcal{F}(v)<+\infty$.
\end{proposition}
\begin{proof}
First we note that according to \eqref{decomposition} we can estimate $\mathcal{F}(v)$
as follows:
$$
\begin{array}{ll}
\mathcal{F}(v)&\dys =\sum_{i=1}^2\int_{\overline{P}}H(d_{1}(x,\partial \Omega))d|D\,w_{x_i}|(x)+
\sum_{i=1}^2\int_{\Omega \setminus {P}}H(d_{1}(x,\partial \Omega))d|D\,v_{x_i}(x)|(x)
\\
&\dys \leq \sum_{i=1}^2\int_{\overline{P}}C\,d|D\,w_{x_i}|(x)+
\sum_{i=1}^2\int_{\bigcup_{j=1}^N \mathcal{W}_j(T_j)}H(d_{1}(x,\partial \Omega))d|D\,v_{x_i}(x)|(x)
\end{array}
$$
where we have used that for some positive constant $C$, $H(d_1(x,\partial \Omega))\leq C$ for all $x \in \overline{\Omega}$.

The first term of the last estimate is finite, due to the choice of $w$. Therefore it is sufficient to estimate 
$$
\sum_{i=1}^2\sum_{j=1}^N\int_{ \mathcal{W}_j(T_j)}H(d_{1}(x,\partial \Omega))d|D\,v_{x_i}(x)|(x)\,.
$$

The construction of $v$, the properties of the rigid motion $\mathcal{W}_j$ and the
duality between the metrics $d_1$ and $d_\infty$ imply that we can
restrict ourselves to prove that there exists a positive constant $M$ such that
\begin{equation}
  \label{eq:2}
  \int_{\mathcal{S}_{T_j}\cup \mathcal{D}_{T_j}^- \cup
    \mathcal{D}_{T_j}^+} H(d_{\infty}(x,\mathcal{W}_j^{-1}(\partial
  \Omega)))d\mathcal{H}^1 \leq M \;\;\; \forall \,j\in \{1,\dots,N\}.
\end{equation}
The last estimate is a consequence of  Lemmata \ref{lemma-S},
\ref{lemma-D+} and \ref{lemma-D-}.
\end{proof}

In the next lemmata we aim to prove \eqref{eq:2}. We will estimate
separately the integral on the sets $\mathcal{S}_{T_j}$,
$\mathcal{D}_{T_j}^-$ and $\mathcal{D}_{T_j}^+$. From now on we will
work on a single triangular domain $T_j$; since the proofs are independent of $j$, we will simply write $T$ instead of $T_j$. Up to a dilatation
and a rigid motion, we can assume in the sequel that
$$
T=\{(x_1,x_2): 0\leq x_1\leq 1, 0\leq x_2\leq h(x_1)\}\,.
$$
with $h(1)=0$, $h(0)\leq 1$ and $h'(t)<0$ for every $t\in (0,1)$. We will denote
by $\gamma$ the curve $ \mathcal{W}_j^{-1}(\partial\Omega)$ and we
observe explicitly that the graph of the function $h$ is a proper subset of
$\gamma$.  We define 
$$
L_n=\left\{x \in T: \frac 1{n+1}<d_{\infty}(x,\gamma)\leq\frac 1{n}\right\}\,,
$$
for $n\in \N$.

 Let $Q\in \mathcal{Q}(T)$ having a not empty intersection with $L_n$ and assume that $Q$ touches
$\gamma$ in $(x_1,h(x_1))$.
Let $\mathtt{e}_n(Q)$, $\mathtt{n}_n(Q)$, $\mathtt{d}_n(Q)$
denote respectively the length of the intersection of the east side,
north side, diagonal parallel to $\ell_{+}$ of $Q$ with $L_n$.
Then it is easily to verify that
\begin{equation}\label{vertical_intersection}
\mathtt{e}_n(Q)\leq h\left(x_1+\frac{1}{n+1}\right)-h\left(x_1+\frac{1}{n}\right)+\frac{1}{n(n+1)}\,,
\end{equation}
\begin{equation}\label{horizontal_intersection}
\mathtt{n}_n(Q)
\leq h^{-1}\left(h\left(x_1\right)+\frac{1}{n+1}\right)-h^{-1}\left(h\left(x_1\right)+\frac{1}{n}\right)+\frac{1}{n(n+1)}
\end{equation}
and
\begin{equation}\label{diagonal_intersection}
\mathtt{d}_n(Q)\leq\frac{1}{n(n+1)}\,.
\end{equation}

\begin{lemma}\label{lemma_number_squares}
Let $N_n$ be the number of squares in $\mathcal{Q}(T)$ which intersect
$L_n$. Then there exists a positive constant $c$ depending only on $\gamma$
such that
\begin{equation}\label{formula_number_squares}
N_n\leq{c}{(n+1)}.
\end{equation}
\end{lemma}

\begin{proof}
Let $(x_1^0,x_1^0)$ be the north-east corner of $q(T)$. 
Define
$$
\dys
x_1^j=h^{-1}\left(x^0_1-\frac{j}{1+n}\right)\,,\,\,\,\,\,\,j=1,\dots,
m\leq (n+1)\,.$$
Fix $i\in \{0,\dots,m\}$ and note that if on the portion of $\gamma$
lying between $(x^{i}_1,h(x^{i}_1))$ and $(x_1^{i+1},h(x_1^{i+1}))$
there are two points, say $p$ and $q$, which are north-east
vertex of squares, $Q_p$ and $Q_q$ in $\mathcal{Q}(T)$, then $L_n\cap
Q_p$ or $L_n\cap Q_q$ is empty.

Consequently  the number of squares in $r(T)$ which intersect $L_n$  can be at most $n+1$.
With the same argument one can prove that the same estimate holds for 
the number of squares in $u(T)$ which intersect $L_n$.
\end{proof}


\begin{lemma}\label{lemma-D+}
There exists a positive constant $M$ such that
\[
 \int_{\mathcal{D}_{T}^+} H(d_{\infty}(x,\gamma))d\mathcal{H}^1 \leq M.
\]
\end{lemma}

\begin{proof}
  
The proof easily follows from (\ref{diagonal_intersection}),
(\ref{formula_number_squares}) and the hypotheses on $H$. Indeed one has
$$
\int_{\mathcal{D}^+_T}H(d_{\infty}(x,\gamma))d\mathcal{H}^1\leq C\sum_{n=1}^{\infty} H\left(\frac{1}{n}\right)\frac{1}{n}<+\infty\,,
$$
for some positive constant $C$ independent of $n$.
\end{proof}

\begin{lemma}\label{lemma-D-}
There exists a positive constant $M$ such that
\begin{equation}\label{stima-D-}
 \int_{\mathcal{D}_{T}^-} H(d_{\infty}(x,\gamma))d\mathcal{H}^1 \leq M.
\end{equation}
\end{lemma}

\begin{proof}
If $L_i$ denotes the total length of
the diagonals parallel to $\ell_{-}$ of the squares added at the
$i$-th step of the construction of $v$, then
\begin{equation}\label{diagonalimeno_step_i}
L_{i} \leq  2 \Haus^1(\gamma).
\end{equation}
Indeed $L_{i}$ is nothing but the length of the non flat parts of the graph of a 
piecewise affine function defined on the interval $(0,1)$ and 
whose derivative is $0$ or $-1$ everywhere but on a finite number of points.

The proof is divided into several steps.
We will denote by $\mathtt{r}_{n}$  the length of the side of $Q_T^{n,\sigma}$ with $\sigma=(\underbrace{r,\dots, r}_{n})$ 
and
$\mathtt{R}_{{n}}=\sum_{i=1}^{{n}}\mathtt{r}_i$;
$\mathtt{u}_{n}$ will denote the length of the side of $Q_T^{n,\sigma}$ with $\sigma=(\underbrace{u,\dots, u}_{n})$.

\vspace{0.2cm}
\emph{Step 1.} Let $h\in C^1((0,1))$.
Assume that there exists $0<\varepsilon<1$ such that $-1+\varepsilon<h'(t)<0$ for every $t\in
(0,1)$. Let ${d}^{-}_n$ be the intersection of $L_n$ with the
diagonal  parallel to $\ell_{-}$ of a square $Q\in \mathcal{Q}(T)$.   
Let $\displaystyle x_1^Q=\min\limits_{(x_1,x_2)\in d_n^{-}} x_1$. Assume that $(x_1^Q,x_2^Q)\in {d}^{-}_n$.  Then 
$$
h\left(x_1^Q+\frac{1}{n}\right)-\frac{1}{n}\leq x_2^Q\leq h\left(x_1^Q+\frac{1}{n+1}\right)-\frac{1}{n+1}\,,
$$ 
since $(x_1^Q,x_2^Q)\in L_n$.
The intersection between ${d}^{-}_n$,  belonging to the line $x_2=-x_1+x_1^Q+x_2^Q$, 
and the lower boundary of $L_n$, that is, $x_2=h\left(x_1+\frac 1n\right)-\frac 1n$, gives 
$$
\begin{array}{ll}
\displaystyle x_1-x_1^Q&\displaystyle =x_2^Q+\frac 1n -h\left(x_1+\frac 1n\right)
\\
&\displaystyle \leq h\left(x_1^Q+\frac 1{n+1}\right)-\frac{1}{n+1}+\frac 1n -h\left(x_1+\frac 1n\right)\,.
\end{array}
$$
Lagrange's theorem implies that
$$
x_1-x_1^Q
\leq |h'(\xi)|\left[x_1-x_1^Q+\frac 1n-\frac{1}{n+1}\right]+\frac 1n-\frac{1}{n+1}
$$
for some $\xi \in \left[x_1^Q+\frac 1{n+1},x_1+\frac 1n\right]$. Using the hypothesis on $h$
we get
$$ 
x_1-x_1^Q\leq (1-\varepsilon)\left(x_1-x_1^Q+\frac{1}{n^2}\right)+\frac{1}{n^2}\,,
$$
that is, 
$$
x_1-x_1^Q\leq \frac{2-\varepsilon}{\varepsilon}\frac{1}{n^2}\,.
$$
This implies that 
$$
\mathcal{H}^1(d_n^{-})\leq\sqrt{2}\frac{2-\varepsilon}{\varepsilon}\frac{1}{n^2}
$$
for a given square $Q$ such that $Q\cap L_n\neq \emptyset$.
Using estimate (\ref{formula_number_squares}) on the number of squares intersecting $L_n$,
the previous inequality gives
$$
\int_{\mathcal{D}_T^-}H(d_{\infty}(x,\gamma))d\mathcal{H}^{1}
\leq c_1(\varepsilon,\gamma)\sum_{n=1}^{\infty} H\left(\frac{1}{n}\right)\frac{1}{n^2}(n+1)
\leq c_2(\varepsilon,\gamma)\sum_{n=1}^{\infty} H\left(\frac{1}{n}\right)\frac{1}{n}\,.
$$
where $c_1(\varepsilon,\gamma)$ and $c_2(\varepsilon,\gamma)$ denote two positive constants depending only on $\gamma$ and $\varepsilon$.
Hypothesis (\ref{hyp_H_crescita}) on $H$ implies that the last sum is
finite.

The case where $-1+\varepsilon<h'(t)<0$ for every $t\in (0,1)$
 can be handled in a
similar way.

\vspace{0.2cm}
\emph{Step 2.} Assume  that there exists $0<\varepsilon<1$ such that $|h'(t)+1|\leq\eps$ for every
$t\in [0,1]$. We are going to prove by induction that
the length of the side $l_{i,k}$, $k=1,\dots,2^i$ of any square $Q_T^{i,k}$ added at the $i$-th step of the covering of Definition \ref{definition_covering} 
satisfies:
\begin{equation}\label{estim-01}
\frac{1}{(2+\eps)^{i}} \leq l_{i,k} \leq \frac{1}{(2-\eps)^{i}}\,\,\,\,\,\forall\, k=1,\dots,2^i\,.
\end{equation}
To to that, observe that the length of the side $\tilde{l}$ of $q(\tilde{T})$, for a given domain $\tilde{T}=\{(x_1,x_2): a<x_1<b, c<x_2<h(x_1)\}$,
 can be estimated by
\begin{equation}\label{first_square_general_domain}
\frac{\max\{b-a, h(a)-c\}}{2+\eps} \leq \tilde{l} \leq \frac{\min\{b-a, h(a)-c\}}{2-\eps}\,.
\end{equation}
Indeed, it is sufficient to compute the intersections between the lines  $x_2=x_1+c-a$ and 
$x_2= (-1\pm\eps) (x_1-b) + c$ or $x_2= (-1\pm\eps) (x_1-a) + h(a)$.
This implies that the length $l_{0,1}$ of the side  of $q(T)$
satisfies
\begin{equation}\label{estim-00}
\frac{1}{2+\eps} \leq l_{0,1} \leq \frac{1}{2-\eps}\,.
\end{equation}
Now, suppose that  estimate \eqref{estim-01} holds for $i-1$. 
At step $i$ we add $2^{i}$ squares and any of these is confined in a  domain belonging to $\mathcal{T}$ 
with one of his sides that coincides with the side of one of the squares added in the previous step.
Thus, to prove that estimate \eqref{estim-01} holds for $i$, 
it suffices to use (\ref{first_square_general_domain})
with $\frac{1}{(2+\eps)^{i-1}}\leq \tilde{l}\leq \frac{1}{(2-\eps)^{i-1}}$.

We need to estimate the $l^{\infty}$-distance from $\gamma$ of
the diagonal ${d}^-$ parallel to $\ell_{-}$ of a square $Q\in \mathcal{Q}(T)$. 
Assume that the length of the side of $Q$ is $l$.
Let $x^Q=(x_1^Q,x_2^Q)\in \gamma$ be the north-east corner of $Q$. Then the $l^{\infty}$-distance from $\gamma$ of  ${d}^-$
is smaller than  the $l^{\infty}$-distance  of  ${d}^-$ from the
piecewise affine function
\[
x_2(x_1)=
\begin{cases}
  (-1+\varepsilon)(x_1-x_1^Q)+x_2^Q\,,  & \textrm{if }\; x_1 > x_1^Q \\
   (-1+\varepsilon)(x_1-x_1^Q)+x_2^Q\,,  & \textrm{if }\; x_1 \leq x_1^Q
\end{cases}.
\]

This gives
\begin{equation}\label{estim-02}
d_{\infty}({d}^-,\gamma)\leq \frac{1+\eps}{2-\eps}l.
\end{equation}

Using the previous estimates and (\ref{diagonalimeno_step_i}) we obtain
\[
\begin{split}
\int_{\mathcal{D}_T^-} H(d_{\infty}(x,\gamma)) \, d\mathcal{H}^{1} &  \leq \sum_{n=1}^{\infty} L_{n} H\left( \frac{1+\eps}{(2-\eps)^{n+1}} \right) \\
&  \leq 2  \Haus^1(\gamma) \sum_{n=1}^{\infty} H\left( \frac{1+\eps}{(2-\eps)^{n+1}} \right)\,.
\end{split}
\]
In order to prove that the last sum is finite, we start by observing that, by the mo\-no\-to\-ni\-ci\-ty of $H$,
\[
H\left( \frac{1+\eps}{(2-\eps)^{n+1}} \right) \leq H\left(\frac{2}{[(2-\eps)^{n}]} \right)\,.
\]
 We now let 
\[
u_{n}=[(2-\eps)^{n}]\;\; \textrm{ and } \;\; a_{n}= H\left( \frac{2}{ n} \right) \frac{1}{n}
\]
and we apply the Schl\"omilch's generalization of the condensation
criterion for series to deduce the desired convergence. We recall that, if
$a_{n}$ is a positive non increasing sequence of real numbers and
$u_{n}$ a strictly increasing sequence of natural numbers such that for some positive constant $C$
\begin{equation}\label{hip-sch-01}
\frac{u_{n+1}-u_{n}}{u_{n}-u_{n-1}}\leq C\,\,\,\forall\,n\in \N
\end{equation}  
 then $\displaystyle\sum_{n=0}^{\infty}a_{n}$ is finite if and only if $\displaystyle
\sum_{n=0}^{\infty} (u_{n+1}-u_{n})a_{u_{n}}$ is finite.
In our case we observe that the convergence of the series $\displaystyle\sum_{n=0}^{\infty}a_{n}$ is assured by  hypothesis (\ref{hyp_H_crescita})
on $H$. Moreover for every $n\geq n(\varepsilon)$ 
\[
\begin{split}
c(\varepsilon)[(2-\varepsilon)^n] \leq u_{n+1}-u_{n} &  \leq (2-\eps)^{n+1}+1-(2-\eps)^{n} \\
 &  \leq \big([(2-\eps)^{n}]+1\big)(1-\eps)+1 \\
 & \leq [(2-\eps)^{n}]+2-\eps
 \end{split}
\] 
for some positive constant $c(\varepsilon)$ independent of $n$. This implies that hypothesis (\ref{hip-sch-01}) is satisfied.
Therefore there exists a positive constant $C$ such that
\[
\begin{split}
\sum_{n=0}^{\infty} H  \left(  \frac{2}{[(2-\eps)^{n}]} \right) &  = \sum_{n=0}^{\infty} H\left( \frac{2}{[(2-\eps)^{n}]} \right) \frac{1}{[(2-\eps)^{n}]}[(2-\eps)^{n}] \\
& \leq C+ \frac {1}{c(\varepsilon)}\sum_{n=n(\varepsilon)}^{\infty} (u_{n+1}-u_{n})a_{u_{n}} < \infty\,.
\end{split}
\]

\vspace{0.2cm}
\emph{Step 3.} 
Assume that $h\in C^1([0,1])$ and fix $0<\eps<1$.
The uniform continuity of $h'$ implies that there exists $\delta>0$
such that if $|t-s|<\delta$ then $|h'(t)-h'(s)|<\varepsilon/4$. Let
$n' \in \N$ be such that the length of the side of $Q_T^{n',\sigma}$  is
less than $\delta$ for every $\sigma \in S_{n'}$. 
Let $(x_1^\sigma,x_2^\sigma)$ be the north-east
corner of $Q_T^{n',\sigma}$.
For any $\sigma
\in S_{n'}$, if $|h'(x_1^\sigma)+1|\leq
\varepsilon/2$, then
\[
|h'(x_1)+1|\leq |h'(x_1)-h'(x_1^\sigma)|+|h'(x_1^\sigma)+1| < \varepsilon
\] 
and if $|h'(x_1^\sigma)+1|> \varepsilon/2$, then
\[
|h'(x_1)+1| \geq |h'(x_1^\sigma)+1|-|h'(x_1)-h'(x_1^\sigma)| > \frac{\varepsilon}{4}.
\] 
Let $N(n')=1+\sum_{i=0}^{n'}2^i=2^{n'+1}$.
 Using (\ref{diagonalimeno_step_i}) we have
$$
\int_{\mathcal{D}_T^-} H(d_{\infty}(x,\gamma)) \, d\mathcal{H}^{1}
\leq 
n'\mathcal{H}^1(\gamma)+
\sum_{i=1}^{N(n')}\int_{\mathcal{D}_{T_i}^-} H(d_{\infty}(x,\gamma))
$$
where, up to dilatations, 
$T_i$, with $i=1\dots N(n')$, 
are triangular domains satisfying the hypotheses of \emph{Step 1} or \emph{2}.
This implies that
$
\int_{\mathcal{D}_T^-} H(d_{\infty}(x,\gamma)) \, d\mathcal{H}^{1}$ is finite.

\vspace{0.2cm}
\emph{Step 4.} 
Assume that
$h\in C^1((0,1))$ and  $\lim\limits_{t\to 0}h'(t)=-\infty$ or 
$\lim\limits_{t\to 1}h'(t)=-\infty$. 
Surely there exists $n'' \in \N$ such that the intervals $[0,\mathtt{u}_{n''}]$, $[\mathtt{R}_{n''},1]$ 
do not contain any $x_1$ such that $h'(x_1)=-1$.
Let $N(n'')=\sum_{i=0}^{n''}2^i-2=2^{n''+1}-3$.
Using (\ref{diagonalimeno_step_i}) we get the following estimate:
$$
\int_{\mathcal{D}_T^-} H(d_{\infty}(x,\gamma)) \, d\mathcal{H}^{1}
\leq 
n''\mathcal{H}^1(\gamma)+
\sum_{i=1}^{N(n'')}\int_{\mathcal{D}_{T_i}^-} H(d_{\infty}(x,\gamma))
$$
where, up to dilatations, 
$T_i$, with $i=1\dots N(n'')$, 
are triangular domains satisfying the hypotheses of \emph{Step 3}.
This implies that
$
\int_{\mathcal{D}_T^-} H(d_{\infty}(x,\gamma)) \, d\mathcal{H}^{1}$ is finite.

\end{proof}

\begin{lemma}\label{lemma-S}
There exists a positive constant $M$ such that
\[
 \int_{\mathcal{S}_{T}} H(d_{\infty}(x,\gamma))d\mathcal{H}^1 \leq M.
\]\end{lemma}

\begin{proof}
 First we observe that, since $T$ is a triangular domain, up to
 inverting the coordinate axes, we can assume that we are in one of
 the following cases:

\vspace{0,2cm}
\emph{Case 1.} There exist two constants $c_1,c_2 >0$ such that $-c_1\leq h'(t)\leq -c_2<0$ for every $t\in [0,1]$.

\vspace{0,2cm}
\emph{Case 2.} There exists a constant $c_1<0$ such that $c_1\leq h'(t)<0$ for every $t\in (0,1)$ and $h'(1)=0$.

\vspace{0,2cm}
\emph{Case 3.} There exists a constant $c_1<0$ such that $c_1\leq h'(t)<0$ for every $t\in (0,1)$ and  $\lim\limits_{t\to 1}h'(t)=-\infty$.

\vspace{0,2cm}
\emph{Step 1.} Assume that we are in the hypotheses of \emph{Case
  1}. We observe that $h$ and $h^{-1}$ are Lipschitz functions, say with Lipschitz constant $\tilde{c}$; then (\ref{vertical_intersection}) and (\ref{horizontal_intersection}) 
together with estimate (\ref{formula_number_squares})
imply that 
$$
\int_{\mathcal{S}_T}H(d_{\infty}(x,\gamma))d\mathcal{H}^1
\leq 
\sum_{n=1}^{\infty}H\left(\frac 1n\right)\frac{\tilde{c}}{n(n+1)}c(n+1)
$$
where $c$ depends only on $\gamma$.
The last sum  is finite due to hypothesis (\ref{hyp_H_crescita}) on $H$.

\vspace{0,2cm}
\emph{Step 2.} Assume that we are in the hypotheses of \emph{Case 2}. It is sufficient to estimate
$$
\int_{\mathcal{S_T}\cap r(T)}H(d_{\infty}(x,\gamma))d\mathcal{H}^1,
$$
since  \emph{Step 1} implies that
$$
\int_{\mathcal{S}_T\cap u(T)}H(d_{\infty}(x,\gamma))d\mathcal{H}^1
$$
is bounded. 
To this purpose, we define the following sequence of points. Let $x_1^0\in (0,1)$ be such that
$h(x_1^0)=x_1^0$. We set
$$\dys x_1^j=h^{-1}\left(x_1^0-\frac jn\right)\,,\,\,\,\,j=1,\dots, [x_1^0 n]\,=:M_n.$$
There exists at most one point $(\tilde{x}^j_1,h(\tilde{x}^j_1))$ with $x_1^{j}\leq \tilde{x}^j_1<x_1^{j+1}$, $j=1,\dots, M_n-1$, 
which is the north-east vertex of a square $Q^j\in \mathcal{Q}(T)$ such that $L_n\cap Q^j$ is not empty. 
For such a square, since $|h'|$ is bounded, one has
\begin{equation}\label{vertical}
\mathtt{e}_n(Q^j)\leq\left[1+\sup_{x_1\in [0,1]}\left|h'\left(x_1\right)\right|\right]\frac{1}{n^2}\leq C_1\frac{1}{n^2}
\end{equation}
for some positive constant $C_1$.
As well
$$
\begin{array}{ll}
\displaystyle
\mathtt{n}_n(Q^j)&
\displaystyle
\leq \frac{1}{n(n+1)}\left[1+\sup_{x_2 \in [h(\tilde{x}^j_1)+\frac{1}{n+1}, h(\tilde{x}^j_1)+\frac{1}{n}]}|(h^{-1})'(x_2)|\right]
\vspace{0.1cm}
\\
&\displaystyle
\leq\frac{1}{n^2}\left[1+\sup_{x_2 \in [h({x}^{j+1}_1), h({x}_1^{j-1})]}\frac{1}{|h'(h^{-1}(x_2))|}\right]
\vspace{0.1cm}
\\
&\displaystyle
\leq\frac{1}{n^2}\left[1+\sup_{x_2 \in [h({x}^{j+1}_1), h({x}_1^{j})]}\frac{1}{|h'(h^{-1}(x_2))|}+
\sup_{x_2 \in [h({x}^{j}_1), h({x}_1^{j-1})]}\frac{1}{|h'(h^{-1}(x_2))|}
\right].
\end{array}
$$

We remark that
$$\displaystyle\sum_{j=1}^{M_n}\frac{1}{n}\left[\sup_{x_2 \in [h(x^j_1),h(x^{j-1}_1)]}\frac{1}{|h'(h^{-1}(x_2))|}\right]$$ 
is a particular Riemann sum for 
$$\displaystyle\int_{\frac 1n}^{x_1^0-\frac 1n}\frac{1}{|h'(h^{-1}(x_2))|}dx_2$$ 
which is finite. 
Therefore there exists $n_0 \in \N$ such that 
\begin{equation}\label{riemann_sum}
\sum_{j=1}^{M_n}\frac{1}{n}\left[\sup_{x_2 \in [h(x^j_1),h(x^{j-1}_1)]}\frac{1}{|h'(h^{-1}(x_2))|}\right]
\leq
\int_{0}^{x_1^0}\frac{1}{|h'(h^{-1}(x_2))|}dx_2+1\leq C_2
\end{equation} 
 for every $n\geq n_0$, for some positive constant $C_2$ independent of $n$.
Estimates (\ref{vertical}) and (\ref{riemann_sum}) give
$$
\begin{array}{l}
\dys
 \int_{\mathcal{S}\cap r(T)}H(d_{\infty}(x,\gamma))d\mathcal{H}^1\leq
\\
\dys \leq 
\dys\sum_{n=n_0}^{\infty} H\left(\frac{1}{n}\right)\sum_{j=1}^{M_n}\left\{\frac{2}{n^2}\sup_{x_2 \in [h(x^j_1),h(x^{j-1}_1)]}\frac{1}{|h'(h^{-1}(x_2))|}+\frac{C_1+1}{n^2}\right\}+
\\
\dys
+\sum_{n=1}^{n_0}
H\left(\frac{1}{n}\right)\mathcal{H}^1(\mathcal{S}\cap r(T)\cap \overline{L_n})
\\
\dys
\leq 
C_2\dys\sum_{n=1}^{\infty} H\left(\frac{1}{n}\right)\frac{1}{n}
+(C_1+1)\sum_{n=1}^{\infty}H\left(\frac{1}{n}\right)\frac{1}{n}+C_3\,,
\end{array}
$$
where we have used that
 for $n\leq n_0$,  $\dys H\left(\frac{1}{n}\right)\leq H\left(\frac{1}{n_0}\right)$ and $\mathcal{H}^1(\mathcal{S}\cap r(T)\cap \overline{L_n})$
 is finite, since $v_{x_i}, i=1,2$ is $SBV_{loc}(T)$. 
The last sum is finite due to hypothesis (\ref{hyp_H_crescita}) on $H$.

\vspace{0,2cm}
\emph{Step 3.}  Assume that we are in the hypotheses of \emph{Case
  3}. As in the previous step it is sufficient to estimate
$$
\int_{\mathcal{S}_T\cap r(T)}H(d_{\infty}(x,\gamma))d\mathcal{H}^1.
$$

Let $x_1^0\in (0,1)$ be such that
$h(x_1^0)=x_1^0$. Let us consider the following sequence of points:
$$x^j_1=x^0_1+\frac jn\,,\,\,\,\,\,\, j=1,\dots,M_n=[n(1-x_1^0)]\,.$$

There exists at most one point $(\tilde{x}^j_1,h(\tilde{x}^j_1))$ with $x_1^{j}\leq \tilde{x}^j_1<x_1^{j+1}$, $j=1,\dots, M_n-1$, 
which is the north-east vertex of a square $Q^j\in \mathcal{Q}(T)$ such that $L_n\cap Q^j$ is not empty.

For such a square, since $|(h^{-1})'|$ is bounded, one has
\begin{equation}\label{vertical_2}
\mathtt{n}_n(Q^j)\leq\left[1+\sup_{x_1\in [0,h(0)]}\left|(h^{-1})'(x_1)\right|\right]\frac{1}{n^2}\leq C_1\frac{1}{n^2}
\end{equation}
for some positive constant $C_1$.
As well
$$
\begin{array}{ll}
\displaystyle
\mathtt{e}_n(Q^j)&
\displaystyle
\leq \frac{1}{n(n+1)}\left[1+\sup_{x_1 \in [\tilde{x}^j_1+\frac{1}{n+1}, \tilde{x}^j_1+\frac{1}{n}]}|h'(x_1)|\right]
\vspace{0.1cm}
\\
&\displaystyle
\leq\frac{1}{n^2}\left[1+\sup_{x_1 \in [{x}^{j}_1, {x}_1^{j+2}]}{|h'(x_1)|}\right]
\vspace{0.1cm}
\\
&\displaystyle
\leq\frac{1}{n^2}\left[1+\sup_{x_2 \in [{x}^{j}_1, {x}_1^{j+1}]}{|h'(x_1)|}+
\sup_{x_2 \in [{x}^{j+1}_1, {x}_1^{j+2}]}\frac{1}{|h'(x_1)|}
\right]\,.
\end{array}
$$

We remark that
$$\displaystyle \sum_{j=1}^{M_n}{\sup_{x_1\in [x^j_1,x^{j+1}_1]}|h'(x_1)|}{\frac 1n}
$$
is a Riemann sum for
$$\displaystyle \int_{x_1^0}^{1-\frac 1n}|h'(t)|dt$$ which is finite. Then there exists ${n}_0 \in \N$ such that 
$$
\frac 1n\sum_{j=1}^{M_n}{\sup_{x_1\in [x^j_1,x^{j+1}_1]}|h'(y)|}<\int_{x_1^0}^1|h'(t)|dt+1\leq C_2
$$
for every $n\geq n_0$, for some positive constant $C_2$ independent of $n$.
Therefore, arguing as in the previous step, we have
$$
\begin{array}{l}
\displaystyle\int_{\mathcal{S}\cap r(T)}H(d_{\infty}(x,\gamma)) d\mathcal{H}^1 
\\
\displaystyle
\leq 
\sum_{n=n_0}^{\infty}H\left(\frac{1}{n}\right)\sum_{j=1}^{M_n}\left\{
\frac{2}{n^2}\sup_{x_1\in [x^j_1,x^{j+1}_1]}|h'(x_1)|+\frac{C_1+1}{n^2}
\right\}
\\
\displaystyle +\sum_{n=1}^{n_0}H\left(\frac{1}{n}\right)\mathcal{H}^1(\mathcal{S}\cap r(T)\cap \overline{L_n})
\\
\displaystyle
\leq C_2
\sum_{n=1}^{\infty}H\left(\frac{1}{n}\right) \frac{1}{n}
+(C_1+1)\sum_{n=1}^{\infty}H\left(\frac{1}{n}\right)\frac{1}{n}+C_3 < \infty\,,
\end{array}
$$
where $C_3$ denotes a positive constant independent of $n$.
\end{proof}

We have therefore shown that the functional $\mathcal{F}$ is well-defined.
We are now in position to show Theorem \ref{main_thm}.
The proof follows quite easily from the direct methods of the calculus
of variations. Note that this technique was already used in
\cite{Champion-Croce}. We start recalling a lemma proved in  \cite{Champion-Croce}. 

\begin{lemma}\label{distancel1}
Let $\Omega$ be an open bounded connected subset of $\R^N$ with Lipschitz boundary. Then
$$
- d_1(\cdot,\dom)\leq v \leq d_1(\cdot,\dom)
\quad on\,\, \overline{\Omega}
$$
for every function $v \in \mathfrak{S}(\Omega)$.
\end{lemma}

\begin{proof}(of Theorem \ref{main_thm})
Let $(v^n)_n \subset \mathcal{E}(\Omega)$ be a minimizing sequence. 
Lemma \ref{distancel1} assures that  $(v^n)_n$ is uniformly bounded
 in $L^{\infty}(\Omega)$.
Moreover $v_n$ is uniformly Lipschitz in $\overline{\Omega}$ since
$\Omega$ is Lipschitz and $| \nabla v^n | \leq \sqrt{2}$ a.e. in $\Omega$ for every $n$.
Therefore, up to a subsequence, $v^n \to v^\infty$ in ${C}_0(\Omega)$
and $v^n \to v^\infty$ weakly* in $W^{1,\infty}(\Omega)$ for some
$v^\infty\in W_0^{1,\infty}(\Omega)$.

We are now going to show that $v^{\infty}$ belongs to $\mathcal{E}(\Omega)$.
Since $(v^n)_n$ is a minimizing sequence, there exists $C>0$ such that $\mathcal{F}(v^n)\leq C$ for every $n \in \N$.
For a fixed $m\in \N$ we can say that
$$
C\geq \mathcal{F}(v^n)\geq
\alpha(m)\sum^2_{i=1}\int\limits_{B_m} d(|D v^n_{x_i}|)(x), \quad \forall \,n \in \N
$$
where
$$
B_m=\left\{x \in \Omega: d_1(x,\partial \Omega)>\frac 1m\right\}\,.
$$
Let us fix $i \in \{1,2\}$.
Note that $\frac{\partial v^n}{\partial x_i}
\in SBV(B_m)$ and takes only the two values $\pm 1$ for every $n$.
This implies that
$$
\norm{\frac{\partial v^n}{\partial x_i}}_{BV(B_m)}\,\,=\,\,
\mathcal{L}^2(B_m)+2\mathcal{H}^{1}({J}_{v^n_{x_i}} \cap B_m)\,\, \leq \,\, \mathcal{L}^2(B_m)+2 C.
$$
We can apply Theorem \ref{compactnessSBV} to the sequence
$\dys \left(\frac{\partial v^n}{\partial x_i}\right)_n$:
\begin{itemize}
\item
hypothesis $i)$ has been verified in  the previous estimate;
\item
hypothesis $ii)$ is verified as $\nabla \frac{\partial v^n}{\partial x_i} = 0$ a.e. in $\Omega$;
\item 
hypothesis $iii)$ can be verified
choosing $f \equiv 1$: in this way
$$
\int_{J_{v^n_{x_i}} \cap B_m} f \left( \left[\frac{\partial v^n}{\partial x_i}\right] \right) d\mathcal{H}^{1}(x)
\,\,\, = \,\,\, \mathcal{H}^{1}(J_{v^n_{x_i}} \cap B_m) \,\,\, \leq \,\,\, C.
$$
\end{itemize}
Consequently $\dys \frac{\partial v^n}{\partial x_i} \to g_i$ weak*
in $BV(B_m)$ for some $g_i \in SBV(B_m)$ .
Since
$\dys \frac{\partial v^n}{\partial x_i} \to \frac{\partial v^\infty}{\partial x_i}$
weak* in $L^\infty(\Omega)$,
we infer that $\dys \frac{\partial v^\infty}{\partial x_i} \in SBV(B_m)$.
Moreover $\dys \frac{\partial v^n}{\partial x_i} \to \frac{\partial v^\infty}{\partial x_i}$
in $L^1(B_m)$
and so $\dys \left|\frac{\partial v^\infty}{\partial x_i}\right|=1$ a.e. in $B_m$.
This being true for every $i$ and for every $B_m$,
we deduce that  $v^\infty$ belongs to $\E(\Omega)$.

To show that $v^{\infty}$ is a minimizer of $\mathcal{F}$,
we remark that
$$
\liminf\limits_{n\to \infty} \int\limits_{B_m} H(d_1{(x,\partial \Omega)})d(|D v^n_{x_i}|)(x)\,
\geq
\int\limits_{B_m} H(d_1{(x,\partial \Omega)})d(|D v^{\infty}_{x_i}|)(x)
$$
for every $B_m$, due to Theorem \ref{semicontinuityBV}.
This implies that
\begin{eqnarray*}
\liminf\limits_{n\to \infty} \int\limits_{\Omega} H(d_1(x,\partial \Omega)\, d(|D v^n_{x_i}|)(x)
& \geq &
\sup_{B_m}
\int\limits_{B_m} H(d_1(x,\partial \Omega)\, d(|D v^{\infty}_{x_i}|)(x)
\\
& = & \int\limits_{\Omega} H(d_1(x,\partial \Omega)\, d(|D v^{\infty}_{x_i}|)(x)\,.
\end{eqnarray*}
Therefore
$\displaystyle \liminf\limits_{n\to \infty} \mathcal{F}(v^n) \geq \mathcal{F}(v^\infty)$,
i.e., $v^{\infty}$ minimizes $\mathcal{F}$.
\end{proof}




\end{document}